\documentclass[12pt]{article}
\usepackage{mathtools,amsthm,amsmath,amsfonts,amssymb,tikz-cd,enumitem, graphicx,mathrsfs,bbm,stmaryrd,xcolor}

\tikzcdset{scale cd/.style={every label/.append style={scale=#1},
    cells={nodes={scale=#1}}}}
\usepackage{url}
\usepackage{esvect}
\usepackage{setspace}
\usepackage[utf8]{inputenc}
\usepackage{CJKutf8}
\counterwithin{equation}{section}
\usepackage[pdfencoding=auto,psdextra,colorlinks=true,linkcolor=blue,citecolor=blue,urlcolor = orange]{hyperref}
\usepackage[margin = 1.25in]{geometry}
\usepackage{kantlipsum}
\allowdisplaybreaks
\newtheorem{theorem}{Theorem}[section]
\newtheorem{definition}[theorem]{Definition}

\newtheorem{remark}[theorem]{Remark}
\newtheorem{proposition}[theorem]{Proposition}

\newtheorem{notations}[theorem]{Notations}
\newcommand{\citep}[1]{\cite{#1}}

\newcommand{\rank}{\mathrm{rank}\hspace{+0.7mm}}

\newcommand{\edmp}{\text{\normalfont End}}
\newcommand{\trace}{\text{\normalfont Tr}}

\newcommand{\AuthorInfo}{}

\newcommand{\AddAuthor}[2]{%
 \appto{\AuthorInfo}{%
   \noindent\textbf{#1}\\#2\par\vspace{1em}
  }%
}

\AddAuthor{}{
Beijing International Center for Mathematical Research, Peking University\\
\texttt{\color{orange}hzhuang@pku.edu.cn}
}

\begin{document}
\title{Bott connection in the mapping cone case}
\date{\today}
\author{Hao Zhuang}
\maketitle
\begin{abstract}
We generalize the concept of Bott connection to the de Rham mapping cone situation. As the main results, when the given closed smooth manifold is foliated, we have the mapping cone version of the Bott vanishing theorem and an adiabatic limit of a family of mapping cone connections on the orthogonal complement of the foliation. The second result shows the necessity of a refined metric if we want to see more interactions between geometry and topology in the mapping cone situation. 
\end{abstract}

\tableofcontents
\section{Introduction}
The Bott connection was defined by Bott \cite[Definition 6.1]{Bott_lecture_notes} for a vanishing theorem of Pontryagin classes. For a specific product of some Pontryagin forms associated to the Bott connection on an integrable subbundle of the tangent bundle, Bott showed that it is equal to $0$ under a constraint on degrees. Then by the Chern-Weil transgression, the cohomology class of this product vanishes. This is the Bott vanishing theorem. 

The Bott connection also appears in studying the relation between the $\hat{A}$-genus and the existence of a metric of positive leaf-wise scalar curvature. In \cite[Theorem 1.1]{liuandzhang1999}, Liu and Zhang found the adiabatic limit of the restriction of the perturbed Levi-Civita connection on the orthogonal complement of the integrable subbundle. The limit is given by the average of the Bott connection and its dual. An important further development of \cite{liuandzhang1999} is a purely differential-geometric proof \cite[Theorem 0.1]{zhang2017positive_scalar_foliation} of the Connes vanishing theorem \cite[Chapter 3 \S 7.$\beta$ Corollary 10]{connes1995noncommutative}.

Now, given a closed smooth manifold $M$, a closed smooth $2$-form $\omega$ on $M$, a Riemannian metric $g$ on $M$, an integrable subbundle $F\subseteq TM$ and its orthogonal complement $F^\perp$ with respect to $g$, we generalize the study of the Bott connection to the de Rham mapping cone situation. 
\begin{notations}\normalfont
     We let $\Omega^k(M)$ be the space of smooth $k$-forms on $M$, $d$ be the de Rham exterior derivative. For any vector bundle $E$ on $M$, we let $\Omega^k(M,E)$ be the space of smooth $E$-valued $k$-forms on $M$. In addition, we use $\Gamma(E)$ to denote the space of smooth sections of $E$. 
\end{notations}
  Since $\omega$ is $d$-closed, we have the following mapping cone cochain complex 
\begin{align}\label{cochain complex}
    d^\omega: \Omega^k(M)\oplus\Omega^{k-1}(M)&\to\Omega^{k+1}(M)\oplus\Omega^{k}(M)\nonumber\\
    (\alpha,\beta)&\mapsto (d\alpha + \omega\wedge\beta, -d\beta).
\end{align}
Its cohomology is called the mapping cone cohomology. In the special case where $\omega$ is symplectic and integral, (\ref{cochain complex}) computes the primitive cohomology of a symplectic manifold \cite{tanaka_tseng_2018, tty3rd, tty1st, tty2nd}. 

Based on \cite{tseng_and_zhou_symplectic_flat_connection_and_twisted_primitive2022, tseng_zhou_symplectic_flat_functional_characteristic_classes2022,tseng_and_zhou_2025mapping_yang_mills}, given a connection $\nabla$ on $F^\perp$ and any $B\in\Gamma(\edmp(F^\perp))$, we have a mapping cone connection 
\begin{align}
    \mathbb{A}: \Omega^{i}(M, F^\perp)\oplus\Omega^{i-1}(M, F^\perp) & \to \Omega^{i+1}(M, F^\perp)\oplus\Omega^{i}(M, F^\perp) \nonumber\\
    (a,b)&\mapsto (\nabla a + \omega\wedge b, Ba-\nabla b).
\end{align}
Following \cite[(1.15)]{wittendeformationweipingzhang} and \cite[Theorem 1.6]{transgression_primitive_2025}, we identify $\mathbb{A}^{2k}$ with an element in $\Omega^{2k}(M, \edmp(F^\perp))\oplus\Omega^{2k-1}(M, \edmp(F^\perp))$, and choose the function
\begin{align}
    f(x) = \dfrac{1}{2}\log\left(1-\dfrac{x^2}{4\pi^2}\right).
\end{align}
     Using the trace map 
     \begin{align}
         \mathbf{Tr}: \Omega^{i}(M, \edmp(F^\perp))\oplus\Omega^{i-1}(M, \edmp(F^\perp))&\to\Omega^{i}(M)\oplus\Omega^{i-1}(M) \nonumber\\
         (a,b)&\mapsto (\mathrm{Tr}(a), \mathrm{Tr}(b)),
     \end{align}
     we obtain the total mapping cone Pontryagin form 
     \begin{align}
    p(\mathbb{A})\coloneqq \exp\left(\mathbf{Tr}(f(\mathbb{A}^2))\right).
\end{align}
By \cite[Theorem 1.6]{transgression_primitive_2025}, it represents a mapping cone cohomology class. 

We write $p(\mathbb{A})$ into
\begin{align}
    (1,0)+p_1(\mathbb{A}) + p_2(\mathbb{A}) + \cdots,\ \text{with} \ p_i(\mathbb{A})\in\Omega^{4i}(M)\oplus\Omega^{4i-1}(M).
\end{align}
The first main result is following Bott type vanishing theorem. 
\begin{theorem}\label{main result 1}
    When $2i_1+\cdots+2i_k > \dim M - \rank F$, the product 
    \begin{align}
        p_{i_1}(\mathbb{A})p_{i_2}(\mathbb{A})\cdots p_{i_k}(\mathbb{A})
    \end{align}
    vanishes in the cohomology of (\ref{cochain complex}). 
\end{theorem}

Similar to \cite[Section 1.7.1]{wittendeformationweipingzhang} and \cite[Section 1.5.1]{weipingzhangnewedition}, the main idea of proving Theorem \ref{main result 1} is to use the mapping cone Bott connection: 
\begin{align}
    \widetilde{\mathbb{A}}^{F^\perp}: \Omega^i(M,F^\perp)\oplus\Omega^{i-1}(M,F^\perp)&\to\Omega^{i+1}(M, F^\perp)\oplus\Omega^i(M, F^\perp)  \nonumber\\
        (a, b)&\mapsto (\widetilde{\nabla}^{F^\perp} a + \omega\wedge b, -\widetilde{\nabla}^{F^\perp} b), 
\end{align}
where $\widetilde{\nabla}^{F^\perp}$ is the original Bott connection \cite[Definition 1.13]{wittendeformationweipingzhang} on $F^\perp$. We let $\widetilde{\mathbb{A}}^{F^\perp,*}$ be its associated dual connection (See Proposition \ref{dual connection verification process}) with respect to the extension of $g$ to $\sum_{i = 0}^{\dim M + 1}\Omega^i(M,F^\perp)\oplus\Omega^{i-1}(M,F^\perp)$ (See (\ref{extend the metric})). 

Now, we let $g^{F^\perp}$ (resp. $g^F$) be the restriction of $g$ on $F^\perp$ (resp. $F$), and $P^\perp$ be the orthogonal projection from $TM$ to $F^\perp$. Then, we let $\varepsilon>0$ be a parameter and $\nabla^{TM, \varepsilon}$ be the Levi-Civita connection on $TM$ associated to 
$g^F\oplus \frac{1}{\varepsilon}g^{F^\perp}$. 

Given $\Phi\in\Gamma(\edmp(TM))$ skew-adjoint with respect to the original $g$, we see 
\begin{align}
    \Phi^\varepsilon \coloneqq P\Phi + \varepsilon P^\perp\Phi 
\end{align}
is skew-adjoint with respect to $g^F\oplus \frac{1}{\varepsilon}g^{F^\perp}$. Then, we have the mapping cone analogue of the Levi-Civita connection 
\begin{align}
     \mathbb{A}^{TM,\varepsilon}: \Omega^i(M, TM)\oplus\Omega^{i-1}(M, TM)&\to \Omega^{i+1}(M, TM)\oplus\Omega^i(M, TM) \nonumber\\
     (a,b)&\mapsto (\nabla^{TM,\varepsilon}a + \omega\wedge b, \Phi^\varepsilon a - \nabla^{TM,\varepsilon} b).
 \end{align}
After extending $P^\perp$ to $\Omega^i(M, TM)\oplus\Omega^{i-1}(M, TM)$ (See (\ref{extend the projection naturally on F perp})), we get $P^\perp\mathbb{A}^{TM,\varepsilon} P^\perp$, i.e., the restriction of 
$\mathbb{A}^{TM,\varepsilon}$ to $F^\perp$. With natural maps 
\begin{align}
    \pi_1: \Omega^i(M,F^\perp)\oplus\Omega^{i-1}(M,F^\perp)&\to\Omega^i(M,F^\perp),\\
    \pi_2: \Omega^i(M,F^\perp)\oplus\Omega^{i-1}(M,F^\perp)&\to\Omega^{i-1}(M,F^\perp),
\end{align}
the second main result is the following adiabatic limit (cf. \cite[Theorem 1.1]{liuandzhang1999}). 
\begin{theorem}\label{main result 2}
For any $(a,b)\in\Omega^i(M,F^\perp)\oplus\Omega^{i-1}(M,F^\perp)$, any $X_1, \cdots, X_{i+1}\in\Gamma(F)$ and $Y\in\Gamma(F^\perp)$, we have 
\begin{align}
    g\left(X_1\lrcorner\cdots\lrcorner X_{i+1}\lrcorner \pi_1\mathbb{A}^{F^\perp,\varepsilon}(a,b), Y\right) &\to \dfrac{1}{2}g\left(X_1\lrcorner\cdots\lrcorner X_{i+1}\lrcorner\pi_1(\widetilde{\mathbb{A}}^{F^\perp}+\widetilde{\mathbb{A}}^{F^\perp,*})(a,b), Y\right)\\
    g\left(X_1\lrcorner\cdots\lrcorner X_{i}\lrcorner \pi_2\mathbb{A}^{F^\perp,\varepsilon}(a,b), Y\right) &\to \dfrac{1}{2}g\left(X_1\lrcorner\cdots\lrcorner X_i\lrcorner\pi_2(\widetilde{\mathbb{A}}^{F^\perp}+\widetilde{\mathbb{A}}^{F^\perp,*})(a,b), Y\right) \label{the second line of main result second}
\end{align}
    uniformly on the closed smooth $M$ when $\varepsilon\to 0$. 
\end{theorem}

This paper is organized as follows. In Section \ref{bott vanishing section}, we construct the mapping cone version of the Bott connection. Then, we plug it into the transgression formula to prove Theorem \ref{main result 1}. In Section \ref{adiabatic limit section}, we explain how we restrict the mapping cone connection associated with the perturbed Levi-Civita connection to $F^\perp$, and how we make sense of the dual of the mapping cone Bott connection. Then, we prove Theorem \ref{main result 2}. Finally, we comment on interacting between geometry and topology in the mapping cone situation.

\section{Bott type vanishing theorem}\label{bott vanishing section}
In this section, we first define the mapping cone Pontryagin class of $F^\perp$ based on \cite[Theorem 1.6]{transgression_primitive_2025}. Then, we recall the Bott connection on $F^\perp$ and define its mapping cone version. Finally, we prove Theorem \ref{main result 1} using the mapping cone Bott connection. The idea of the proof is from \cite[Section 1.7.1]{wittendeformationweipingzhang} and \cite[Section 1.5.1]{weipingzhangnewedition}. 

Recall that given any $B\in\Gamma(\edmp(F^\perp))$ and any connection 
\begin{align}
    \nabla: \Omega^i(M,F^\perp)\to\Omega^{i+1}(M,F^\perp)
\end{align}
on $F^\perp$, we have the mapping cone connection \cite{tseng_and_zhou_symplectic_flat_connection_and_twisted_primitive2022, tseng_zhou_symplectic_flat_functional_characteristic_classes2022, tseng_and_zhou_2025mapping_yang_mills} associated to $\nabla$ and $B$: 
\begin{align}
    \mathbb{A}: \Omega^{i}(M, F^\perp)\oplus\Omega^{i-1}(M, F^\perp) & \to \Omega^{i+1}(M, F^\perp)\oplus\Omega^{i}(M, F^\perp) \nonumber\\
    (a,b)&\mapsto (\nabla a + \omega\wedge b, Ba-\nabla b).
\end{align}
Following \cite[(1.15)]{wittendeformationweipingzhang}, we let 
\begin{align}
    f(x) = \dfrac{1}{2}\log\left(1-\dfrac{x^2}{4\pi^2}\right)
\end{align}
and define the associated mapping cone version of the total Pontryagin form as
\begin{align}\label{Pontryagin form definition}
    p(\mathbb{A})\coloneqq \exp\left(\textbf{Tr}(f(\mathbb{A}^2))\right).
\end{align}
\begin{remark}
    \normalfont 
    In (\ref{Pontryagin form definition}), we identify $\mathbb{A}^{2k}$ with an element in 
    \begin{align}
        \Omega^{2k}(M,\edmp(F^\perp))\oplus\Omega^{2k-1}(M,\edmp(F^\perp)). 
    \end{align}
    Let Tr be the trace map on $\Omega^i(M,\edmp(F^\perp))$, we extend it to 
    \begin{align}
    \textbf{Tr}: \Omega^{2k}(M,\edmp(E))\oplus\Omega^{2k-1}(M,\edmp(E))&\to\Omega^{2k}(M)\oplus\Omega^{2k-1}(M) \nonumber\\
    (a,b)&\mapsto (\trace(a), \trace(b)).
    \end{align}
    Also, we need the product structure
    \begin{align}
        (\alpha, \beta)\cdot (\gamma,\delta) = (\alpha\wedge\gamma, \beta\wedge\gamma + (-1)^{\deg(\alpha)}\alpha\wedge\delta)
    \end{align}
    on $\sum_{i = 0}^{\dim M + 1}\Omega^i(M)\oplus\Omega^{i-1}(M)$. Here, $\deg(\alpha)$ is the degree of $\alpha$. See \cite[Section 1.5]{bgv} and \cite[(2.5), Example 4.3]{transgression_primitive_2025} for details. 
\end{remark}
By \cite[Theorem 1.6]{transgression_primitive_2025}, $p(\mathbb{A})$ is $d^\omega$-closed. Also, given another mapping cone connection $\mathbb{A}'$, we have 
\begin{align}\label{transgression}
    p(\mathbb{A}) - p(\mathbb{A}') = d^\omega(\alpha, \beta)
\end{align}
for some $\alpha, \beta\in\Omega^*(M)$. Thus, $p(\mathbb{A})$ defines a mapping cone cohomology class.

By decomposing with respect to degrees, we see that 
\begin{align}
    p(\mathbb{A}) = (1,0) + p_1(\mathbb{A}) + p_2(\mathbb{A}) + \cdots 
\end{align}
with the $i$-th component $p_i(\mathbb{A})\in\Omega^{4i}(M)\oplus\Omega^{4i-1}(M)$ being $d^\omega$-closed. 

 Because of (\ref{transgression}), we just need to find an $\mathbb{A}'$ such that 
\begin{align}\label{the technique by Bott}
    p_{i_1}(\mathbb{A}')p_{i_2}(\mathbb{A}')\cdots p_{i_k}(\mathbb{A}') = (0, 0) 
\end{align}
(cf. \cite[(1.29)]{wittendeformationweipingzhang}). Then, we automatically have 
\begin{align}
    p_{i_1}(\mathbb{A})p_{i_2}(\mathbb{A})\cdots p_{i_k}(\mathbb{A}) = d^\omega(\alpha, \beta)
\end{align}
for some $\alpha, \beta\in\Omega^*(M)$ and finish the proof of Theorem \ref{main result 1}. 

Now we construct this $\mathbb{A}'$. Let $\nabla^{TM}$ be the Levi-Civita connection associated to the metric $g$ on $TM$. Let $F^\perp$ be the orthogonal complement of $F$ in $TM$ with respect to the Riemannian metric $g$.  
Let $P: TM\to F$ (resp. $P^\perp: TM\to F^\perp$) be the orthogonal projection to $F$ (resp. $F^\perp$). Following \cite[Definition 1.13]{wittendeformationweipingzhang}, we define 
    \begin{align}
        \widetilde{\nabla}^{F^\perp}: \Gamma(TM)\times\Gamma(F^\perp)&\to\Gamma(F^\perp)
    \end{align}
    by letting 
    \begin{align}
    \widetilde{\nabla}^{F^\perp}_X U = 
    \begin{cases}
        P^\perp [X, U] \text{\ \ when\ }X\in\Gamma(F), U\in\Gamma(F^\perp), \\
        P^\perp\nabla^{TM}_X U \text{\ \ when\ }X\in\Gamma(F^\perp), U\in\Gamma(F^\perp).
    \end{cases}
    \end{align}
    Then, we extend it to $\Omega^i(M, F^\perp)$: 
    \begin{align}
        \widetilde{\nabla}^{F^\perp}: \Omega^i(M, F^\perp)&\to\Omega^{i+1}(M, F^\perp) \nonumber\\
        \alpha\otimes U &\mapsto d\alpha\otimes U + (-1)^i \alpha\wedge\widetilde{\nabla}^{F^\perp} U. 
    \end{align}
Let $g^F$ and $g^{F^\perp}$ be the restrictions of $g$ on $F$ and $F^\perp$ respectively. Then, we define the mapping cone Bott connection:
\begin{definition}\normalfont
    We call
    \begin{align}\label{Bott connection cone case expression}
        \widetilde{\mathbb{A}}^{F^\perp}: \Omega^i(M,F^\perp)\oplus\Omega^{i-1}(M,F^\perp)&\to\Omega^{i+1}(M, F^\perp)\oplus\Omega^i(M, F^\perp)  \nonumber\\
        (a, b)&\mapsto (\widetilde{\nabla}^{F^\perp} a + \omega\wedge b, -\widetilde{\nabla}^{F^\perp} b)
    \end{align}
    the mapping cone Bott connection. 
\end{definition}

Using the identification in \cite[(2.5)]{transgression_primitive_2025}, we have: 
\begin{proposition}
    The $(\widetilde{\mathbb{A}}^{F^\perp})^{2k}$ identifies with  
    \begin{align}
        \left((\widetilde{\nabla}^{F^\perp})^{2k}, 0\right)\in \Omega^{2k}(M,\edmp(F^\perp))\oplus\Omega^{2k-1}(M,\edmp(TM))
    \end{align}
    for all $k\geqslant 1$. 
\end{proposition}

Let $F^{\perp,*}$ be the dual bundle of $F^{\perp}$. We recall the description of $(\widetilde{\nabla}^{F^\perp})^{2k}$ induced by \cite[Lemma 1.14]{wittendeformationweipingzhang}: 
\begin{proposition}
    Under the identification, we have 
    \begin{align}\label{degree of the curvature with B equal to 0}
        (\widetilde{\nabla}^{F^\perp})^{2k} \in \Gamma\left(\Lambda^k F^{\perp,*}\right)\wedge\Omega^*(M, \edmp(F^\perp))
    \end{align}
    for all $k\geqslant 1$.  
\end{proposition}

\begin{remark}
\normalfont
    We need the integrability of the bundle $F$ to ensure (\ref{degree of the curvature with B equal to 0}) because the proof of \cite[Lemma 1.14]{wittendeformationweipingzhang} relies on $[X_1, X_2]\in \Gamma(F)$ for all $X_1, X_2\in\Gamma(F)$.
\end{remark}

By the degree of each $p_i(\widetilde{\mathbb{A}}^{F^\perp})$ together with (\ref{degree of the curvature with B equal to 0}), we see that 
\begin{align}
    p_{i_1}(\widetilde{\mathbb{A}}^{F^\perp})\cdots p_{i_k}(\widetilde{\mathbb{A}}^{F^\perp}) = (\alpha, 0)
\end{align}
with $\alpha\in \Gamma\left(\Lambda^{2(i_1+\cdots + i_k)}F^{\perp,*}\right) \wedge \Omega^*(M)$.
If we have 
\begin{align}
    2i_1+\cdots+2i_k > \rank (F^{\perp, *}) = \dim M - \rank F,
\end{align}
then we find
\begin{align}
    p_{i_1}(\widetilde{\mathbb{A}}^{F^\perp})\cdots p_{i_k}(\widetilde{\mathbb{A}}^{F^\perp}) = (0, 0).
\end{align}
This means $\widetilde{\mathbb{A}}^{F^\perp}$ satisfies the requirement of (\ref{the technique by Bott}). Theorem \ref{main result 1} is now proved. 

\section{Adiabatic limit}\label{adiabatic limit section}
In this section, we let $\Phi\in\Gamma(\edmp(TM))$ be a skew-adjoint section of endomorphism with respect to the original $g$. Let $\varepsilon>0$ be the parameter. 
 After rescaling the Riemannian metric $g$ into 
 \begin{align}
     g^\varepsilon = g^F \oplus \dfrac{1}{\varepsilon} g^{F^\perp}, 
 \end{align}
 we see that the subbundles $F$ and $F^\perp$ and the projections $P$ and $P^\perp$ remain unchanged for all $\varepsilon>0$, while
 \begin{align}
     \Phi^\varepsilon \coloneqq P\Phi + \varepsilon P^\perp\Phi
 \end{align}
 is skew-adjoint with respect to $g^\varepsilon$. 
 Let $\nabla^{TM,\varepsilon}$ be the Levi-Civita connection on $TM$ associated to $g^\varepsilon$. We have the mapping cone analogue of the Levi-Civita connection: 
 \begin{align}
     \mathbb{A}^{TM,\varepsilon}: \Omega^i(M, TM)\oplus\Omega^{i-1}(M, TM)&\to \Omega^{i+1}(M, TM)\oplus\Omega^i(M, TM) \nonumber\\
     (a,b)&\mapsto (\nabla^{TM,\varepsilon}a + \omega\wedge b, \Phi^\varepsilon a - \nabla^{TM,\varepsilon} b)
 \end{align}
 associated to $\nabla^{TM,\varepsilon}$ and $\Phi^\varepsilon$. We will verify the adiabatic limit in Theorem \ref{main result 2} of the restriction of $\mathbb{A}^{TM,\varepsilon}$ to $F^\perp$, 
 generalizing \cite[Theorem 1.1]{liuandzhang1999}.
 
  Both $P^\perp$ and $P$ extend to $\Omega^i(M, TM)\oplus\Omega^{i-1}(M, TM)$ in this way: For all $\alpha\in\Omega^i(M)$, $\beta\in\Omega^{i-1}(M)$, and $X, Y\in\Gamma(TM)$: 
\begin{align}
    P^\perp(\alpha\otimes X, \beta\otimes Y) =\ & (\alpha\otimes P^\perp(X), \beta\otimes P^\perp(Y)), \label{extend the projection naturally on F perp}\\
    P(\alpha\otimes X, \beta\otimes Y) =\ & (\alpha\otimes P(X), \beta\otimes P(Y)).
\end{align}
\begin{definition}
    We call 
    \begin{align}
        \mathbb{A}^{F^\perp,\varepsilon}\coloneqq P^\perp\mathbb{A}^{TM,\varepsilon}P^\perp: \Omega^i(M, F^\perp)\oplus\Omega^{i-1}(M, F^\perp)&\to\Omega^{i+1}(M, F^\perp)\oplus\Omega^i(M, F^\perp)
    \end{align}
    the restriction of $\mathbb{A}^{TM,\varepsilon}$ to $F^\perp$. 
\end{definition}

As \cite[(1.33)]{wittendeformationweipingzhang}, we let $\widetilde{\nabla}^{F^\perp,*}$ be the dual connection on $F^\perp$ of the Bott connection $\widetilde{\nabla}^{F^\perp}$ in the following sense: For any $U, V\in\Gamma(F)$, 
\begin{align}
    d(g(U, V)) = g(\widetilde{\nabla}^{F^\perp}U, V) + g(U, \widetilde{\nabla}^{F^\perp, *}V).
\end{align}
Now, we extend $g$ to $\sum_{i = 0}^{\dim M + 1}\Omega^i(M,F^\perp)\oplus\Omega^{i-1}(M,F^\perp)$
by 
\begin{align}\label{extend the metric}
    &\ g((\alpha\otimes U, \beta\otimes V), (\gamma\otimes X, \delta\otimes Y)) \nonumber\\
    \coloneqq & \left( g(U,X)\alpha\wedge\gamma + (-1)^{i} g(U, Y)\alpha\wedge\delta,\ g(V, X)\beta\wedge\gamma \right)
\end{align}
for all $\alpha\in\Omega^i(M), \beta\in\Omega^{i-1}(M), \gamma\in\Omega^j(M), \delta\in\Omega^{j-1}(M)$ and $U, V, X, Y\in\Gamma(F^\perp)$. 
Notice that (\ref{extend the metric}) gives us a pair of $F^\perp$-valued forms instead of a function. 

Taking degrees into consideration, the dual $\widetilde{\mathbb{A}}^{F^\perp,*}$ of $\widetilde{\mathbb{A}}^{F^\perp}$ should satisfy
\begin{align}
    d^\omega g\left((a, b), (r, s)\right) = g\left(\widetilde{\mathbb{A}}^{F^\perp}(a, b), (r, s)\right) + (-1)^i g\left((a, b), \widetilde{\mathbb{A}}^{F^\perp,*}(r,s)\right)
\end{align}
for all $(a, b)\in\Omega^i(M, F^\perp)\oplus\Omega^{i-1}(M, F^\perp)$ and $(r, s)\in\Omega^j(M, F^\perp)\oplus\Omega^{j-1}(M, F^\perp)$. 

\begin{proposition}\label{dual connection verification process}
The mapping cone connection $\widetilde{\mathbb{A}}^{F^\perp,*}$ is given by 
    \begin{align}\label{extended dual connection}
        \widetilde{\mathbb{A}}^{F^\perp,*}: \Omega^i(M,F^\perp)\oplus\Omega^{i-1}(M,F^\perp)&\to\Omega^{i+1}(M, F^\perp)\oplus\Omega^i(M, F^\perp)  \nonumber\\
        (a, b)&\mapsto (\widetilde{\nabla}^{F^\perp,*} a + \omega\wedge b, -\widetilde{\nabla}^{F^\perp,*} b).
    \end{align}
\end{proposition}
\begin{proof}
    Let $\alpha\in\Omega^i(M), \beta\in\Omega^{i-1}(M), \gamma\in\Omega^j(M), \delta\in\Omega^{j-1}(M)$, and $U, V, X, Y\in \Gamma(F^\perp)$. Then, we find 
    \begin{align}
           & g\left(\widetilde{\mathbb{A}}^{F^\perp}(\alpha\otimes U, \beta\otimes V), (\gamma\otimes X, \delta\otimes Y)\right)  \nonumber\\
        =\ & \left(d\alpha\wedge\gamma\wedge g(U,X) + (-1)^{i+j}\alpha\wedge\gamma \wedge g(\widetilde{\nabla}^{F^\perp} U, X) + \omega\wedge\beta\gamma\wedge g(V,X), \right.\nonumber\\
       &\ \  -d\beta\wedge\gamma\wedge g(V,X) + (-1)^{i+j}\beta\wedge\gamma\wedge g(\widetilde{\nabla}^{F^\perp} V, X) + (-1)^{i+1}d\alpha\wedge\delta\wedge g(U,Y) \nonumber\\
       &\ \ \ \ \ \ \ \ \ \ \ \ \ \ \ \ \ \ \left.+ (-1)^j\alpha\wedge\delta\wedge g(\widetilde{\nabla}^{F^\perp} U, Y) + (-1)^{i+1}\omega\wedge\beta\wedge\delta\wedge g(V,Y)\right).
    \end{align}
    Using (\ref{extended dual connection}), we find 
    \begin{align}
        & g\left((\alpha\otimes U, \beta\otimes V), \widetilde{\mathbb{A}}^{F^\perp,*}(\gamma\otimes X, \delta\otimes Y)\right)  \nonumber\\
        =\ & \left(\alpha\wedge d\gamma \wedge g(U,X) + (-1)^j\alpha\wedge\gamma\wedge g(U,\widetilde{\nabla}^{F^\perp,*}X) + \alpha\wedge\omega\wedge\delta \wedge g(U,Y), \right.\nonumber\\
       &\ \  -(-1)^i\alpha\wedge d\delta \wedge g(U, Y) + (-1)^{i+j}\alpha\wedge\delta\wedge  g(U,\widetilde{\nabla}^{F^\perp,*}Y)+\beta\wedge d\gamma \wedge g(V,X) \nonumber\\
       &\ \ \ \ \ \ \ \ \ \ \ \ \ \ \ \ \ \ \left.+ (-1)^j\beta\wedge\gamma \wedge g(V, \widetilde{\nabla}^{F^\perp,*}X) + \beta\wedge\omega\wedge\delta \wedge g(V,Y)\right).
    \end{align}
    Thus, we have 
    \begin{align}
        & g\left(\widetilde{\mathbb{A}}^{F^\perp}(\alpha\otimes U, \beta\otimes V), (\gamma\otimes X, \delta\otimes Y)\right) + (-1)^i g\left((\alpha\otimes U, \beta\otimes V), \widetilde{\mathbb{A}}^{F^\perp,*}(\gamma\otimes X, \delta\otimes Y)\right) \nonumber\\
        =\ & \left(d(\alpha\wedge\gamma\wedge g(U,X)) + \omega\wedge\beta\wedge\gamma\wedge g(V,X) + (-1)^i\omega\wedge\alpha\wedge\delta\wedge g(U,Y),\right. \nonumber\\
        & \ \ \ \ \ \ \ \ \ \ \ \ \ \ \ \ \ \ \ \ \ \ \ \ \ \ \ \ \ \ \ \ \left.-d(\beta\wedge\gamma\wedge g(V,X)) - (-1)^i d(\alpha\wedge\delta\wedge g(U,Y)\right) \nonumber\\
        =\ & d^\omega g\left((\alpha\otimes U, \beta\otimes V), (\gamma\otimes X, \delta\otimes Y)\right).
    \end{align}
    This shows that $\widetilde{\mathbb{A}}^{F^\perp,*}$ is the dual mapping cone connection we need. 
\end{proof}

Now we go back to $\mathbb{A}^{F^\perp,\varepsilon}$. For any $(\gamma\otimes U, \delta\otimes V)\in\Omega^i(M, F^\perp)\oplus\Omega^{i-1}(M, F^\perp)$, we have 
\begin{align}\label{image under the mapping cone connection restricted on F perp}
    & \mathbb{A}^{F^\perp,\varepsilon}(\gamma\otimes U, \delta\otimes V) \nonumber\\
    =\ & (d\gamma\otimes U + (-1)^i\gamma\wedge P^\perp(\nabla^{TM, \varepsilon}U) + (\omega\wedge\delta)\otimes V, \nonumber\\
   & \ \ \  \varepsilon \gamma\otimes P^\perp\Phi(U) - d\delta\otimes V - (-1)^{i-1}\delta\wedge P^\perp(\nabla^{TM,\varepsilon}V)). 
\end{align}
\begin{remark}
    \normalfont From (\ref{image under the mapping cone connection restricted on F perp}), we see $\mathbb{A}^{F^\perp,\varepsilon}$ is also a mapping cone connection. 
\end{remark}

Recall the following adiabatic limit \cite[Theorem 1.15]{wittendeformationweipingzhang}: 
\begin{theorem}[Liu-Zhang \cite{liuandzhang1999}, 2001]\label{adiabatic limit original}
    For all $U, V\in\Gamma(F^\perp)$ and $X\in \Gamma(F)$, when $\varepsilon\to 0$, we have 
    \begin{align}
        g\left(\nabla^{TM, \varepsilon}_X U, V\right) \to \dfrac{1}{2}g\left(\widetilde{\nabla}^{F^\perp}_X U+\widetilde{\nabla}^{F^\perp, *}_X U, V\right)
    \end{align}
    uniformly on $M$. 
\end{theorem}
Therefore by Theorem \ref{adiabatic limit original}, for all $Y\in\Gamma(F^\perp)$ and $X_1, \cdots, X_{i+1}\in\Gamma(F)$, the first entry in (\ref{image under the mapping cone connection restricted on F perp}) satisfies
\begin{align}\label{limit 1}
    & \lim_{\varepsilon\to 0} g\left(X_1\lrcorner X_2\lrcorner \cdots X_{i+1}\lrcorner\left(d\gamma\otimes U + (-1)^i\gamma\wedge P^\perp(\nabla^{TM, \varepsilon}U) + (\omega\wedge\delta)\otimes V\right), Y\right)  \nonumber\\
    =\ & g\left(X_1\lrcorner X_2\lrcorner \cdots X_{i+1}\lrcorner\Big(d\gamma\otimes U + (-1)^i\gamma\wedge \dfrac{1}{2}(\widetilde{\nabla}^{F^\perp}U + \widetilde{\nabla}^{F^\perp,*}U) + (\omega\wedge\delta)\otimes V\Big), Y\right), 
\end{align}
and the second entry in (\ref{image under the mapping cone connection restricted on F perp}) satisfies
\begin{align}\label{limit 2}
    & \lim_{\varepsilon\to 0} g\left(X_1\lrcorner X_2\lrcorner \cdots X_{i}\lrcorner\left(\varepsilon \gamma\otimes P^\perp\Phi(U) - d\delta\otimes V - (-1)^{i-1}\delta\wedge P^\perp(\nabla^{TM,\varepsilon}V)\right), Y\right)  \nonumber\\
    =\ & g\left(X_1\lrcorner X_2\lrcorner \cdots X_{i}\lrcorner\Big(- d\delta\otimes V - (-1)^{i-1}\delta\wedge \dfrac{1}{2}(\widetilde{\nabla}^{F^\perp}V + \widetilde{\nabla}^{F^\perp,*}V)\Big), Y\right).
\end{align}
Meanwhile, by (\ref{Bott connection cone case expression}) and (\ref{extended dual connection}), we find
\begin{align}\label{limit 3}
       & \left(\dfrac{1}{2}\widetilde{\mathbb{A}}^{F^\perp} + \dfrac{1}{2}\widetilde{\mathbb{A}}^{F^\perp,*}\right)(\gamma\otimes U, \delta\otimes V) \nonumber\\
    =\ &  \left(d\gamma\otimes U + (-1)^i\gamma\wedge \dfrac{1}{2}(\widetilde{\nabla}^{F^\perp}U + \widetilde{\nabla}^{F^\perp,*}U) + (\omega\wedge\delta)\otimes V, \right. \nonumber\\
    &\ \ \ \ \ \ \ \ \ \ \ \ \ \ \ \ \ \ \left.- d\delta\otimes V - (-1)^{i-1}\delta\wedge \dfrac{1}{2}(\widetilde{\nabla}^{F^\perp}V + \widetilde{\nabla}^{F^\perp,*}V)\right).
\end{align}
 Theorem \ref{main result 2} is now proved by (\ref{limit 1}), (\ref{limit 2}), and (\ref{limit 3}). 

\begin{remark}
\normalfont
    Our computation uses the local expression. In fact, the compactness of $M$ ensures that the convergence in Theorem \ref{main result 2} holds true uniformly on $M$.  
\end{remark}
 
We end this paper by the following comment. Based on \cite{tanaka_tseng_2018}, we may view
\begin{align}
    (0,1)\in\Omega^1(M)\oplus\Omega^0(M)
\end{align}
intuitively as a globally defined angular form $\theta$ on a circle bundle $S$ over $M$. By taking the dual, we obtain a globally defined vector field $\Theta$ on $S$. This vector field automatically gives a foliation whose codimension in $S$ is equal to $\dim M$. Now, here is the question, when we are given the integrable $F\subseteq TM$, is $F\oplus \Theta$ an integrable subbundle of $TS$?  

For the situation in this paper, we do not have a rigorous precise expression of $\Theta$, so we just intuitively view $F\oplus\Theta$ as an integrable subbundle of $TS$. The contractions in (\ref{the second line of main result second}) should be intuitively understood as 
\begin{align}
    \cdots X_1\lrcorner X_2\lrcorner\cdots X_i\lrcorner\Theta\lrcorner \cdots. 
\end{align}
In addition, the only reasonable metric that we can put on
$
    \Omega^i(M, F^\perp)\oplus\Omega^{i-1}(M, F^\perp)
$
is the product metric. The intuitions are rigorously realized through the mapping cone structure together with this product metric.

However, what about the symplectic situation? When $\omega$ is symplectic and integral, the objects $S$, $\theta$, and $\Theta$ all have rigorous constructions. If we want to study the interaction between geometry and topology on $S$, this $F\oplus \Theta\subseteq TS$ becomes an interesting starting point. In addition, although the product metric brings a lot of convenience on the topological side, in this symplectic situation, $\Theta$ involves both the horizontal direction and the fiber direction, and the product metric is not enough anymore. For more interesting geometry, we need a more refined metric. A recent work \cite{refined_metric_on_mapping_cone} has already applied such a metric to the symplectic Morse theory. 

\section*{Acknowledgments}
I thank Prof. Hongzhi Liu, Prof. Xiaobo Liu, Prof. Yu Qiao, Prof. Xiang Tang, Prof. Li-Sheng Tseng, Prof. Hang Wang, Prof. Shanwen Wang, Dr. Jinxuan Chen, and Dr. Zijing Wang for relevant discussions. Also, I thank Beijing International Center for Mathematical Research for the vibrating working environment.  

\bibliographystyle{abbrv}
\bibliography{mybib.bib}
\AuthorInfo
\end{document}